\documentclass[11pt, oneside]{amsart}
\usepackage{amssymb}
\usepackage{mathtools}
\usepackage{geometry}
\usepackage[colorlinks]{hyperref}
\usepackage{natbib}
\usepackage{tablists}

\newtheorem{theorem}{Theorem}
\newtheorem{proposition}[theorem]{Proposition}
\newtheorem{lemma}[theorem]{Lemma}

\theoremstyle{definition}

\theoremstyle{remark}
\newtheorem{remark}[theorem]{Remark}
\numberwithin{equation}{section}

\begin{document}

\title{Maximal Fuchsian subgroups of the $d=2$ Bianchi group}
\author{Anthony Lee}
\address{KAIST Department of Mathematics, Daejeon, South Korea}
\email{anthonylee@kaist.ac.kr}

\begin{abstract}
    Let $\Gamma$ denote the $d = 2$ Bianchi group $\operatorname{PSL}(2,\mathbb{Z}[\sqrt{-2}])$. We give an explicit description of all conjugacy classes of maximal nonelementary Fuchsian subgroups of $\Gamma$ as integral orders of certain indefinite quaternion algebras over $\mathbb{Q}$. Using this description, we also provide the covolumes corresponding to each conjugacy class. As an application, we compute the limit $\lim_{x\to\infty} \frac{\Pi(x)}{x}$ where $\Pi(x)$ counts the number of primitive totally geodesic immersed surfaces in $\Gamma\backslash\mathbb{H}^3$ with area less than $x$.    
\end{abstract}

\maketitle

\section{Introduction}

We identify the ideal points $\partial\mathbb{H}^3$ of hyperbolic 3-space $\mathbb{H}^3$ with the Riemann sphere $\mathbb{C} \cup \{\infty\}$. Under this identification, the action of $\operatorname{PSL}(2,\mathbb{C})$ on the boundary $\mathbb{C}$ can be described as linear fractional transformations. Let $\mathcal{O}_2$ denote the ring $\mathbb{Z}[\sqrt{-2}]$ and let $\Gamma$ be the $d = 2$ Bianchi group $\operatorname{PSL}(2,\mathcal{O}_2)$. To avoid cumbersome repetition, we will call a \[\text{maximal nonelementary Fuchsian subgroup of $\Gamma$}\] an \textit{$F$-subgroup}. In this paper, we provide an explicit description of the conjugacy classes of $F$-subgroups of $\Gamma$ as $\mathbb{Z}$-orders in certain indefinite quaternion algebras, and through this, we provide the volume formulae corresponding to each $F$-subgroup.

Such an explicit description is possible because up to conjugacy, every $F$-subgroup can be described as a subgroup of $\Gamma$ which, through the action described above, stabilizes a unique circle $\mathcal{C}$ in $\mathbb{C}$ whose explicit equation is known. The equation of such a circle can be represented uniquely in a \textit{reduced form} \[A|z|^2 + B\overline{z} + \overline{B}z + C = 0,\] where $A, C \in \mathbb{Z}$ and $B \in \mathcal{O}_2$, and furthermore \[\operatorname{gcd}(A,\operatorname{Re}(B), \operatorname{Im}(B), C) = 1.\] For a reduced form, we define its discriminant $D$ to be $B\overline{B} - AC$, which should be a positive integer. The explicit equations defining the circles $\mathcal{C}$ corresponding to each $\operatorname{PGL}(2,\mathcal{O}_2)$-conjugacy class of $F$-subgroups can be found using the results of \cite{Vulakh_1991}.

On the other hand, results of \cite{e5932844-3dc8-37ed-b4a8-b4d26ff81211} provide the exact number $n_2(D)$ of the $\Gamma$-conjugacy classes of $F$-subgroups whose corresponding circles $\mathcal{C}$ have discriminant $D$, and these numbers depend only on the congruence class of $D$ modulo some fixed integer. Note that $n_2(D) < \infty$ is ensured by \cite{Maclachlan_Reid_1991}. Since these numbers agree with the number of formulae found above, all conjugacy classes of $F$-subgroups are fixed under the $\operatorname{PGL}(2,\mathcal{O}_2)$-action. 

\begin{remark}As mentioned in \cite{e5932844-3dc8-37ed-b4a8-b4d26ff81211}, such an agreement may not occur for other Bianchi groups in general.\end{remark}

Using these explicit equations, we then find a description of their corresponding $F$-subgroups as $\mathbb{Z}$-orders in quaternion algebras. Let \[Q=(-2,D)_{\mathbb{Q}}\] be the indefinite quaternion algebra  over $\mathbb{Q}$ with standard basis $1, i, j, ij$ where $i^2 = -2$, $j^2 = D$ and $ij = -ji$. Here, $D \in \mathbb{Z}_{>0}$ is the discriminant of the $\mathcal{C}$ corresponding to the $F$-subgroup we will realize. It is possible to find a $\mathbb{Z}$-order $\mathcal{M}$ in $Q$ with respect to some matrix representation \[\rho : Q\otimes\mathbb{C} \to \operatorname{M}_2(\mathbb{C})\] such that $\rho(\mathcal{M}^1)$ becomes the $F$-subgroup we wanted to represent. We denote as $\mathcal{M}^1$ the group of elements of reduced norm 1 in $\mathcal{M}$. \begin{theorem}\label{mainthm} Let $D$ be the discriminant of the equation of the circle corresponding to an $F$-subgroup. Then each conjugacy class of $F$-subgroups can be represented as the group of reduced norm 1 elements of the following $\mathbb{Z}$-orders in $Q$:\begin{enumerate}
  \item $\mathbb{Z}[1,i,j,ij]$ for any positive integer $D$,
  \item $\mathbb{Z}[1,i,\frac{1+j}{2},\frac{i+ij}{2}]$ whenever $D \equiv 1 \mod 4$,

  \item $\mathbb{Z}[1,i,\frac{3i+j}{4},\frac{2+ij}{4}]$ whenever $D \equiv 2 \mod 16$,
  \item $\mathbb{Z}[1,i,\frac{i+j+ij}{4},\frac{2+2i+ij}{4}]$ whenever $D \equiv 6 \mod 16$,

  \item $\mathbb{Z}[1,i,\frac{i+j}{2},\frac{ij}{2}]$ whenever $D \equiv 2 \mod 4$,
  \item $\mathbb{Z}[1,i,\frac{1+j+ij}{2},\frac{i+ij}{2}]$ whenever $D \equiv 3 \mod 4$.
\end{enumerate}
\end{theorem}

Furthermore, this explicit description allows us to compute the covolumes of the corresponding $F$-subgroups, which are given below. Throughout the paper, the symbol $\left(\frac{-2}{p}\right)$ is defined to be zero for $p = 2$, and will denote the Legendre symbol for $p>2$.

\begin{theorem}\label{volumes}
  Let $\mathcal{M}_{(i)}$ be the (i)-th $\mathbb{Z}$-order in Theorem \ref{mainthm}. For simplicity, let $F(D)$ denote $D\prod_{p|D}\left(1+\left(\frac{-2}{p}\right)p^{-1}\right)$ for positive integers $D$. Then the covolumes of the groups $\mathcal{M}_{(i)}^1/\{\pm 1\} \subset \operatorname{PSL}(2,\mathbb{R})$ are given as $c\pi F(D)$, where the values of $c$ in each case are given as follows:\begin{itemize}
    \item[] $\mathcal{M}_{(1)}$: $c = 1$ when $D \equiv 0 \mod 8$ and $c = 2$ otherwise,

    \item[] $\mathcal{M}_{(2)}$: $c = 1$ when $D \equiv 1 \mod 8$ and $c = \frac{1}{3}$ when $D \equiv 5 \mod 8$,

    \item[] $\mathcal{M}_{(3)},\mathcal{M}_{(4)}$: $c = \frac{1}{6}$ regardless of the congruence class of $D$,
    \item[] $\mathcal{M}_{(5)}$: $c = \frac{1}{2}$,
    
    \item[] $\mathcal{M}_{(6)}$: $c = 1$ when $D \equiv 3 \mod 8$, and $c = \frac{1}{3}$ when $D \equiv 7 \mod 8$.
  \end{itemize}
\end{theorem}

\begin{remark}
Let $K = \mathbb{Q}(\sqrt{-d})$ for positive squarefree $d$. The $\mathbb{Z}$-orders and corresponding covolumes of maximal nonelementary Fuchsian subgroups of $\operatorname{PSL}(2,\mathcal{O}_K)$ are given in \cite{Maclachlan_Reid_1991} for the $d = 1$ case, and in \cite{JUNG2019160} for $d \equiv 3 \mod 4$ under the assumption that the ideal class group of $K$ does not contain any element of order 4. Furthermore a prime geodesic theorem for such $d$ has been given in \cite{JUNG2019160}. However, the precise methods of \cite{JUNG2019160} except Lemma \ref{analyticlemma} cannot be extended to the $d = 2$ case, as there are orders with basis elements having denominators a higher power of 2 (namely, $\mathcal{M}_{(3)}$ and $\mathcal{M}_{(4)}$). We also give a detailed explanation on the 2-adic reduced norm group calculation in Section \ref{nrd} which complements the volume formula discussion of both \cite{Maclachlan_Reid_1991} and \cite{JUNG2019160}.
\end{remark}

A simple application of the results of Theorem \ref{volumes} and Lemma \ref{analyticlemma} yields an asymptotic formula for the growth of primitive immersed totally geodesic surfaces in the hyperbolic 3-fold $\Gamma\backslash\mathbb{H}^3$.  \begin{theorem}Let $\Pi(x)$ denote the number of primitive immersed totally geodesic surfaces in $\Gamma\backslash\mathbb{H}^3$ with area less than $x$. Then,\begin{equation}\label{asymp}\lim_{x\to\infty}\frac{\Pi(x)}{x} = \frac{45}{16\pi}\cdot\prod_{p>2}\left(1-\frac{1}{p}+\frac{1}{p+\left(\frac{-2}{p}\right)}\right)\end{equation} where $\left(\frac{-2}{p}\right)$ is the Legendre symbol.
\end{theorem}

Together with the results of \cite{JUNG2019160} and \cite{Maclachlan_Reid_1991}, we have obtained a list of precise prime geodesic theorems for surfaces in the $d = 1, 2$ and $d \equiv 3 \mod 4$ Bianchi orbifolds under the assumption that the ideal class group of $\mathbb{Q}(\sqrt{-d})$ contains no order 4 element.

\begin{remark}
  The most general classification result for Bianchi groups, to date, is in \cite{inbook} which assumes that the class group contains no order 4 element. Furthermore, for $d \not\equiv 3 \mod 4$, one can see for example through the results of \cite{e5932844-3dc8-37ed-b4a8-b4d26ff81211} that already the $d=5$ case can have up to 9 conjugacy classes having the same discriminant, and moreover the congruence condition on the discriminant is given modulo 200. Hence, even for small $d$ it may be impractical to find the $\mathbb{Z}$-orders and covolumes using the methods of this paper.
\end{remark}

\subsection*{Acknowledgements} The author is deeply grateful to Prof.\ Alan Reid for his encouraging feedback and valuable suggestions regarding publication. The author would also like to thank Prof.\ Minju Lee for helpful discussions throughout the course of this project, and for suggesting improvements to this paper. The author also thanks Prof.\ Wansu Kim for comments on the local analysis at $p=2$. This project was greatly influenced by the work of J.\ Jung \cite{JUNG2019160}. This research was supported by grant No.\ G04250060.

\section{Obtaining the Integral Orders}

We start with a list of explicit formulae corresponding to $\operatorname{PGL}(2,\mathcal{O}_2)$-conjugacy classes of $F$-subgroups. This is a direct consequence of the classification result in \cite{Vulakh_1991}.\begin{proposition}[\cite{Vulakh_1991}, Theorem 3]\label{vulakh}
  Each $\operatorname{PGL}(2,\mathcal{O}_2)$-conjugacy class of $F$-subgroups of $\Gamma$ is represented uniquely by the group $\operatorname{Stab}(f_{k,c}(z), \Gamma)$ with $c \in \mathbb{Z}$, where the $f_{k,c}(z)$ are given as $(k)$: \begin{tabenum}
    \tabenumitem $|z|^2 + c$,
    \tabenumitem $2|z|^2 + \overline{z} + z + 2c$,

    \tabenumitem $4|z|^2 - \sqrt{-2}\overline{z} + \sqrt{-2}z + 4c$,
    \tabenumitem $4|z|^2 + (2-\sqrt{-2})\overline{z} + (2+\sqrt{-2})z + 4c$,
    
    \tabenumitem $2|z|^2 - \sqrt{-2}\overline{z} + \sqrt{-2}z + 2c$,
    \tabenumitem $2|z|^2 + (1-\sqrt{-2})\overline{z} + (1+\sqrt{-2})z + 2c$,
  \end{tabenum}all presented in reduced form. Here, $\operatorname{Stab}(f,\Gamma)$ is the subgroup of $\Gamma$ whose elements stabilize the circle $\mathcal{C}$ defined by $f$.
\end{proposition}

Observe that the discriminant $D$ in each case is $-c$, $1-4c$, $2-16c$, $6-16c$, $2-4c$, and $3-4c$, respectively. Furthermore from the results of \cite{e5932844-3dc8-37ed-b4a8-b4d26ff81211}, one can calculate that the numbers of $\Gamma$-conjugacy classes $n_2(D)$ are given as\[n_2(D) = \begin{cases}
  1 & \text{if $D \equiv 0$ mod 4},\\
  3 & \text{if $D \equiv 2$ or $6$ mod 16},\\
  2 & \text{otherwise,}
\end{cases}\] and these numbers agree with the number of formulae in \ref{vulakh} having corresponding discriminants. To find an explicit description of the $\mathbb{Z}$-orders associated to the stabilizers of each of the six equations, we start with the simplest case of $|z|^2 + c$ and use computations in this case to obtain a description of other cases.

\begin{lemma}\label{embedding}
  Let $Q$ denote the quaternion algebra $(-2,D)_{\mathbb{Q}}$ with standard basis $1$, $i$, $j$, $ij$. Define a matrix embedding $\rho : Q\otimes\mathbb{C} \to \operatorname{M}_2(\mathbb{C})$ of $Q$ by \[i \mapsto \begin{bmatrix}
    \sqrt{-2} & 0\\
    0 & -\sqrt{-2}
  \end{bmatrix},\hspace{1em} j \mapsto \begin{bmatrix}
    0 & D\\
    1 & 0
  \end{bmatrix},\hspace{1em} ij \mapsto \begin{bmatrix}
    0 & D\sqrt{-2}\\
    -\sqrt{-2} & 0
  \end{bmatrix}.\] Let $\mathcal{M}$ be the $\mathbb{Z}$-order $\mathbb{Z}[1,i,j,ij]$ inside $Q$. Then, the group $P\rho(\mathcal{M}^1) := \rho(\mathcal{M}^1)/\{\pm 1\} \subset \Gamma$ consists precisely of elements whose action stabilizes the circle $|z|^2-D = 0$.
\end{lemma}
\begin{proof}
  First notice that elements of the form \[\begin{bmatrix}
    a & Db\\
    \overline{b} & \overline{a}
  \end{bmatrix} \in \operatorname{PSL}(2,\mathbb{C})\] fix the circle $|z|^2 - D = 0$. More explicitly, $z$ is sent to \[z^\prime := \frac{az+Db}{\overline{b}z + \overline{a}}\] and it is straightforward to verify that $|z^\prime|^2 - D = 0$. Given an element $\alpha := t + xi + yj + zij$ in $\mathcal{M}$, we have \[\rho(\alpha) = \begin{bmatrix}
    t + x\sqrt{-2} & Dy + Dz\sqrt{-2}\\
    y - z\sqrt{-2} & t-x\sqrt{-2}
  \end{bmatrix}\] and observe that the image of reduced norm 1 elements are precisely inside $\operatorname{PSL}(2,\mathbb{C})$. Furthermore as $t,x,y,z \in \mathbb{Z}$, we conclude that \[P\rho(\mathcal{M}^1) = \operatorname{Stab}(|z|^2-D,\Gamma).\]
\end{proof}

Next, we find some $T \in \operatorname{PSL}(2,\mathbb{C})$ whose action on $\partial\mathbb{H}^3$ sends the circle $f_{1,c}$ to $f_{k,c}$. This allows us to easily describe the stabilizer groups using the fact that \[\operatorname{Stab}(f_{k,c},\operatorname{PSL}(2,\mathbb{C})) = T^{-1}\cdot\operatorname{Stab}(f_{1,c},\operatorname{PSL}(2,\mathbb{C}))\cdot T,\] and that taking the intersection with $\Gamma$ yields $\operatorname{Stab}(f_{k,c},\Gamma)$. Notice that this group is precisely $\rho(\mathcal{M}_{(k)}^1)/\{\pm 1\}$ according to our notation.

\begin{lemma}
  Let $A|z|^2 + B\overline{z} + \overline{B}z + C = 0$ be in reduced form. Then the element \[T = \begin{bmatrix}
    \sqrt{A} & \frac{B}{\sqrt{A}}\\
    0 & \frac{1}{\sqrt{A}}
  \end{bmatrix} \in \operatorname{PSL}(2,\mathbb{C})\] sends the circle to the one defined by $|z|^2 = D$, where $D$ is the discriminant of the reduced form.
\end{lemma}
\begin{proof}
  By direct calculation, verify that $(Az + B)(A\overline{z} + \overline{B}) - D = 0$.
\end{proof}

The embedding $\rho$, as defined in Lemma \ref{embedding}, sends reduced norm 1 elements of the form $t + xi + yj + zij$ for $t,x,y,z \in \mathbb{R}$ to elements of the group $\operatorname{Stab}(f_{1,c},\operatorname{PSL}(2,\mathbb{C}))$. Thus we may conjugate the matrix representation $\rho$ to obtain another matrix representation $\rho^\prime = T^{-1}\rho T$ of $Q$, and this gives rise to $\operatorname{Stab}(f_{k,c},\operatorname{PSL}(2,\mathbb{C}))$. Explicitly, $\rho^\prime$ is given as follows:\[\rho^\prime(i) = \begin{bmatrix}
  \sqrt{-2} & \frac{2\sqrt{-2}B}{A}\\
  0 & -\sqrt{-2}
\end{bmatrix},\hspace{1em} \rho^\prime(j) = \begin{bmatrix}
  -B & \frac{D-B^2}{A}\\
  A & B
\end{bmatrix},\hspace{1em} \rho^\prime(ij)= \begin{bmatrix}
  B\sqrt{-2} & \frac{D+B^2}{A}\sqrt{-2}\\
  -A\sqrt{-2} & -B\sqrt{-2}
\end{bmatrix}.\] To find the $\mathbb{Z}$-order of $Q$ corresponding to $\operatorname{Stab}(f_{i,c},\Gamma)$, it suffices to find the conditions on $t$, $x$, $y$, $z$ such that the matrix \[\rho^\prime(t+xi+yj+zij)\] has entries in $\mathcal{O}_2$. As $A$ appears in the denominator and hence affects the integrality conditions, we will split cases according to when $A = 2$, and when $A = 4$.

\subsection{The orders $\mathcal{M}_{(2)}$, $\mathcal{M}_{(5)}$, $\mathcal{M}_{(6)}$} The element $t + xi + yj + zij$ via $\rho^\prime$ is sent to \[\begin{bmatrix}
  t+x\sqrt{-2} - yB + zB\sqrt{-2} & xB\sqrt{-2} + y\frac{D-B^2}{2} + z\frac{D+B^2}{2}\sqrt{-2}\\
  2y-2z\sqrt{-2} & t-x\sqrt{-2} + yB - z\sqrt{-2}B
\end{bmatrix}\] where we calculate each case for $B = 1$, $-\sqrt{-2}$ and $1-\sqrt{-2}$. Each case yields the matrices \[\begin{bmatrix}
  t-y+(x+z)\sqrt{-2} & y\frac{D-1}{2} + (x+\frac{D+1}{2}z)\sqrt{2}\\
  2y - 2z\sqrt{-2} & t+y - (x+z)\sqrt{-2}
\end{bmatrix}\] for $B = 1$, \[ \begin{bmatrix}
  t+2z+(x+y)\sqrt{-2} & 2x + y\frac{D+2}{2} + z(\frac{D-2}{2})\sqrt{-2}\\
  2y - 2z\sqrt{-2} & t-2z-(x+y)\sqrt{-2}
\end{bmatrix}\] for $B = -\sqrt{-2}$ and \[\begin{bmatrix}
  (t-y+2z) + (x+y+z)\sqrt{-2} & 2x + y\frac{D+1}{2} + 2z + (x+y+z\frac{D-1}{2})\sqrt{-2}\\
  2y-2z\sqrt{-2} & t+y-2z - (x+y+z)\sqrt{-2}
\end{bmatrix}\] for $B = 1-2\sqrt{-2}$, respectively. It is immediate that in all three case we must have $t,x,y,z \in \frac{1}{2}\mathbb{Z}$, whence we may rewrite $t+xi+yj+zij$ as \[\frac{1}{2}(a+bi+cj+dij)\] for $a,b,c,d\in\mathbb{Z}$.

For the first case $\mathcal{M}_{(2)}$, we obtain the two integral relations \[
  a + c \equiv b+ d \equiv 0 \mod 2
\] and hence the order has a $\mathbb{Z}$-basis consisting of the elements \[\frac{1+j}{2},\hspace{0.5em} \frac{i}{2},\hspace{0.5em} j ,\hspace{0.5em} ij.\] This can be seen immediately by replacing each $a,b,c,d$ by integer parameters. Through identical methods, in the case $\mathcal{M}_{(5)}$ we obtain the relations \[
  a \equiv b+c \equiv 0 \mod 2
\] and for $\mathcal{M}_{(6)}$ we have \[
  a+c \equiv b+c+d \equiv 0 \mod 2,
\] so in each case we have $\mathbb{Z}$-bases consisting of the elements \[1,\hspace{0.5em} \frac{i+j}{2},\hspace{0.5em} j,\hspace{0.5em} \frac{ij}{2}\] and \[\frac{1+j+ij}{2},\hspace{0.5em} \frac{i+ij}{2},\hspace{0.5em} j,\hspace{0.5em} ij,\] respectively. Following the convention of \cite{Maclachlan_Reid_1991} which includes $1$ and $i$ as generating elements of the $\mathbb{Z}$-orders, we obtain \begin{equation}\mathbb{Z}\left[1, i, \frac{1+j}{2}, \frac{i+ij}{2}\right],\hspace{1em} \mathbb{Z}\left[1,i,\frac{i+j}{2},\frac{ij}{2}\right],\hspace{1em} \mathbb{Z}\left[1,i,\frac{1+j+ij}{2},\frac{i+ij}{2}\right]\end{equation} as explicit descriptions of $\mathcal{M}_{(2)}$, $\mathcal{M}_{(5)}$, and $\mathcal{M}_{(6)}$, respectively.

\subsection{The orders $\mathcal{M}_{(3)}$, $\mathcal{M}_{(4)}$} For $\mathcal{M}_{(3)}$, the element $t + xi + yj + zij$ is sent to \[\begin{bmatrix}
  t+2z + (x+y)\sqrt{-2} & x+y\frac{D+2}{4} + z\frac{D-2}{4}\sqrt{-2}\\
  4y-4z\sqrt{-2} & t-2z-(x+y)\sqrt{-2}
\end{bmatrix}\] and we obtain $t \in \frac{1}{2}\mathbb{Z}$ and $x,y,z\in \frac{1}{4}\mathbb{Z}$, keeping in mind the congruence relations of $D$ enforced by the equations in \ref{vulakh}. Writing $t+xi+yj+zij$ as \[\frac{a}{2} + \frac{1}{4}(bi + cj + dij)\] for $a,b,c,d\in\mathbb{Z}$, we obtain the integral relations \[b + c \equiv 0 \mod 4\] and \[a + d \equiv 0 \mod 2,\] which leads to a $\mathbb{Z}$-basis consisting of the elements \[\frac{2+ij}{4},\hspace{0.5em}\frac{ij}{2},\hspace{0.5em}\frac{i+3j}{4},\hspace{0.5em}j.\] By a change of $\mathbb{Z}$-basis we obtain an explicit description \begin{equation}\mathcal{M}_{(3)} =\mathbb{Z}\left[1,i,\frac{3i+j}{4},\frac{2+ij}{4}\right].\end{equation}

For $\mathcal{M}_{(4)}$, we obtain \[\begin{bmatrix}
  t-2y+2z+(x+y+2z)\sqrt{-2} & x+\frac{D-2}{4}y + 2z + (x+y+\frac{D+2}{4}z)\sqrt{-2}\\
  4y-4z\sqrt{-2} & t+2y-2z-(x+y+2z)\sqrt{-2}
\end{bmatrix}\] from which follows the integral relations \[a + c + d\equiv 0 \mod 2,\] and \[b + c + 2d \equiv 0 \mod 4,\] following the convention of the case $\mathcal{M}_{(3)}$. This leads to a $\mathbb{Z}$-basis consisting of the elements \[\frac{2+2i+ij}{4}, \hspace{0.5em}\frac{i+j+ij}{4},\hspace{0.5em}\frac{ij}{2}, \hspace{0.5em}i\] and after a change of basis, leads to the explicit description \begin{equation}\mathcal{M}_{(4)} = \mathbb{Z}\left[1,i,\frac{i+j+ij}{4},\frac{2+2i+ij}{4}\right].\end{equation}

\section{Volume Formulae}

Let $N(\mathcal{M})$ denote the absolute value of the reduced discriminant of the order $\mathcal{M}$. From \cite{Voight1} we obtain a volume formula for any $\mathbb{Z}$-order $\mathcal{M}$ in $Q$ as\begin{equation}\label{volform}\operatorname{vol}(P\rho(\mathcal{M}^1)\backslash\mathbb{H}^2) = \frac{\pi N(\mathcal{M})}{3}\prod_{p|N(\mathcal{M})}\lambda(\mathcal{M},p)\cdot[\mathbb{Z}_p^\times: \operatorname{nrd}(\mathcal{M}_p^\times)]^{-1}\end{equation} where $P\rho(\mathcal{M}^1)$ is understood to be in $\operatorname{PSL}(2,\mathbb{R})$, and the symbol $\lambda(\mathcal{M},p)$ is defined to be \[\lambda(\mathcal{M},p) = \frac{1-p^{-2}}{1-\left(\frac{\mathcal{M}}{p}\right)p^{-1}}\] where $\left(\frac{\mathcal{M}}{p}\right)$ is the Eichler symbol of $\mathcal{M}$ at $p$. Let $\mathcal{O}$ be the $\mathbb{Z}$-order $\mathbb{Z}[1,i,j,ij]$ in $Q$. By standard computations, the absolute value of the reduced discriminant of $\mathcal{O}$ is $8D$. Using the relation \[N(\mathcal{O}) = [\mathcal{M}:\mathcal{O}]\cdot N(\mathcal{M}),\] we obtain the reduced discriminants of the remaining cases $\mathcal{M}$ of Theorem \ref{mainthm}. For the orders (2), (5), and (6), we have $[\mathcal{M}:\mathcal{O}] = 4$ and hence $N(\mathcal{M}) = 2D$. For (3) and (4) we have $[\mathcal{M}:\mathcal{O}] = 16$, and hence $N(\mathcal{M}) = \frac{D}{2}$.

Observe that for $p \neq 2$, the $\mathbb{Z}_p$-order $\mathcal{M}_p$ for orders $\mathcal{M}$ of Theorem \ref{mainthm} are equal to $\mathcal{O}_p$ as all denominators that occur are powers of 2 and hence invertible in such $\mathbb{Z}_p$. Let $\alpha = t + xi + yj + zij \in \mathcal{O}_p$. As \[\operatorname{nrd}(\alpha) \equiv t^2 + 2x^2 \mod p,\] for any $p\neq 2$ we may always find $\alpha$ such that $\operatorname{nrd}(\alpha) \notin (\mathbb{Z}_p^\times)^2$. To see this, consider the cases where 2 is and is not a quadratic residue modulo $p$, separately. If $2$ is a quadratic residue modulo $p$, the expression is equivalent to a sum of two squares, and modulo $p$ this attains all possible values. If 2 is not a quadratic residue modulo $p$, the value 2 itself would suffice. 

Together with the inclusion of the groups \[(\mathbb{Z}_p^\times)^2 \subset \operatorname{nrd}(\mathcal{O}_p^\times) \subset \mathbb{Z}_p^\times,\] and the fact that for $p \neq 2$ the subgroup $(\mathbb{Z}_p^\times)^2$ of squares in $\mathbb{Z}_p^\times$ is of index 2, we conclude that \[[\mathbb{Z}_p^\times : \operatorname{nrd}(\mathcal{M}_p^\times)] = 1\] for every $p \neq 2$. 

A simplification can be obtained also for the Eichler symbols at $p \neq 2$. As $\mathcal{M}_p = \mathcal{O}_p$ we have \[\left(\frac{\mathcal{M}}{p}\right) = \left(\frac{\mathcal{O}}{p}\right),\] and furthermore by \cite{Voight2} the value of $\left(\frac{\mathcal{O}}{p}\right)$ can be obtained by computing the Kronecker symbol \[\left(\frac{\Delta(\alpha)}{p}\right)\] for $\alpha \in \mathcal{O}$ where \[\Delta(\alpha) := \operatorname{trd}(\alpha)^2 - 4\operatorname{nrd}(\alpha).\] Namely, for a placeholder $\epsilon\in \{-1,0,1\}$, the Eichler symbol of an order $\mathcal{M}$ is \[\left(\frac{\mathcal{M}}{p}\right) = \epsilon \text{  if and only if } \left(\frac{\Delta(\alpha)}{p}\right) = 0 \text{ or }\epsilon\] for every $\alpha \in \mathcal{M}$. As in our case \[\Delta(\alpha) = -8x^2 + 4Dy^2 + 8Dz^2,\] and since we are looking at primes $p\neq 2$ that divide $D$, the corresponding Kronecker symbol turns out to be \[\left(\frac{-2x^2}{p}\right) = 0 \text{ or } \left(\frac{-2}{p}\right)\] depending on the value of $x$ modulo $p$. Therefore we may replace the Eichler symbol $\left(\frac{\mathcal{M}}{p}\right)$ with the Legendre symbol $\left(\frac{-2}{p}\right)$ in the case $p \neq 2$, and the expression of $\lambda(\mathcal{M},p)$ simplifies into \[\lambda(\mathcal{M},p) = 1+\left(\frac{-2}{p}\right)p^{-1}\] by cancelling out the term $1-\left(\frac{-2}{p}\right)p^{-1}$. In summary, Equation \ref{volform} simplifies into \begin{equation}
  \operatorname{vol}(P\rho(\mathcal{M}^1)\backslash\mathbb{H}^2) = \frac{\pi N(\mathcal{M}) \cdot\lambda(\mathcal{M},2)}{3\cdot[\mathbb{Z}_2^\times:\operatorname{nrd}(\mathcal{M}_2^\times)]}\prod_{p|N(\mathcal{M})}\left(1+\left(\frac{-2}{p}\right)p^{-1}\right)
\end{equation} whenever $2 | N(\mathcal{M})$, noticing that $\left(\frac{-2}{2}\right)=0$, and into \begin{equation}
  \operatorname{vol}(P\rho(\mathcal{M}^1)\backslash\mathbb{H}^2) = \frac{\pi N(\mathcal{M})}{3}\prod_{p|N(\mathcal{M})} \left(1+\left(\frac{-2}{p}\right)p^{-1}\right)
\end{equation} when $2 \nmid N(\mathcal{M})$. Therefore, except the cases (3) and (4) of Theorem \ref{mainthm} where 2 does not divide $N(\mathcal{M})$, we must calculate both the Eichler symbols at $p = 2$, and the indices of the reduced norm groups at $p = 2$ to obtain the precise volume formulae.

\subsection{Computing the Eichler symbol of $\mathcal{M}$ at $p=2$}\label{eichler} We compute the Eichler symbols of the orders $\mathcal{M}$ of (1), (2), (5) and (6) in Theorem \ref{mainthm} by computing $\left(\frac{\Delta(\alpha)}{2}\right)$ for $\alpha \in \mathcal{M}$. Recall that we denote the (i)-th order of Theorem \ref{mainthm} as $\mathcal{M}_{(\text{i})}$. For $\alpha = t+ ix +jy +ijz \in \mathcal{M}_{(1)}$, we have \[\Delta(\alpha) = -8x^2+4Dy^2+8Dz^2\] which is divisible by 2 regardless of the value of $D$ and hence $\left(\frac{\Delta(\alpha)}{2}\right) = 0$. Therefore, \[\lambda(\mathcal{M}_{(1)},2) = \frac{3}{4}.\] For $\alpha = t + xi + y\frac{1+j}{2} + z\frac{i+ij}{2} \in \mathcal{M}_{(2)}$ we have \[\Delta(\alpha) = -8x^2 - 8xz - 2z^2 + Dy^2 + 2Dz^2,\] but as $D \equiv 1 \mod 4$ this value is congruent to $Dy^2$ modulo 8. Hence the symbol may depend on whether $D \equiv 1$ or $5$ mod 8. In the case $D \equiv 1\mod 8$, this becomes \[\Delta(\alpha) \equiv y^2 \mod 8,\] and $\left(\frac{\Delta(\alpha)}{2}\right)$ can either be $0$ or $1$, depending on the value of $y$. When $D \equiv 5 \mod 8$, the symbol $\left(\frac{\Delta(\alpha)}{2}\right)$ is 0 or $-1$ depending on the value of $y$. In summary, \[\lambda(\mathcal{M}_{(2)},2) = \begin{cases}
  \frac{3}{2} & \text{when}\hspace{0.5em} D \equiv 1 \mod 8,\\
  \frac{1}{2} & \text{when}\hspace{0.5em} D \equiv 5 \mod 8.
\end{cases}\] For $\alpha = t + xi + y\frac{i+j}{2} + z\frac{ij}{2}\in\mathcal{M}_{(5)}$, we have \[\Delta(\alpha) = -2(2x+y)^2 + Dy^2 + 2Dz^2,\] but in this case $D \equiv 2 \mod 4$, so $\Delta(\alpha)$ is even. Therefore, \[\lambda(\mathcal{M}_{(5)},2) = \frac{3}{4}.\] For $\alpha = t + xi + y\frac{1+j+ij}{2} + z\frac{i+ij}{2} \in \mathcal{M}_{(6)}$, we have \[\Delta(\alpha) \equiv (D-2)y^2 + 4yz + 2z^2 \mod 8,\] which is even when $y$ is. However if $y$ is odd, then depending on whether $D \equiv 3$ or $7 \mod 8$, the value of $\Delta(\alpha)$ can be either \[\Delta(\alpha) \equiv y^2+4yz+2z^2 \text{, or  }  5y^2 + 4yz+2z^2,\] modulo 8. First assume that $D \equiv 3 \mod 8$. In the case that $z$ is even, we obtain \[\Delta(\alpha) \equiv y^2 \mod 8,\] which is 1 modulo 8 as $y$ is odd. In the case that $z$ is odd, we obtain \[\Delta(\alpha) \equiv (y+2z)^2-z^2-z^2 \equiv -1 \mod 8,\] keeping in mind that the only odd square modulo 8 is 1. Hence $\Delta(\alpha) \equiv \pm 1 \mod 8$, and \[\lambda(\mathcal{M}_{(6)},2) = \frac{3}{2}\hspace{1em}\text{when $D \equiv 3\mod 8$}.\] On the other hand, for $D \equiv 7 \mod 8$, we have either \[\Delta(\alpha) \equiv 5y^2 \equiv -3\mod 8\] or \[\Delta(\alpha) \equiv (2y+z)^2+y^2+z^2 \equiv 3 \mod 8,\] depending on whether $z$ is even or odd. Hence $\Delta(\alpha) \equiv \mp 3\mod 8$, and \[\lambda(\mathcal{M}_{(6)},2) = \frac{1}{2} \hspace{1em}\text{when $D \equiv 7 \mod 8$.}\] This completes the computation of the Eichler symbols at $p=2$.

\subsection{Computing the Index of the Reduced Norm Group at $p = 2$}\label{nrd} Unlike odd primes $p$ where the index of the square unit group in $\mathbb{Z}_p^\times$ is of index 2, the square unit group in $\mathbb{Z}_2^\times$ is of index 4, and more precisely $\mathbb{Z}_2^\times/(\mathbb{Z}_2^\times)^2$ is isomorphic to the Klein four group $V_4$. This is because there are three intermediate subgroups \[1,3+8\mathbb{Z}_2, \hspace{1em} 1,5+8\mathbb{Z}_2,\hspace{1em} 1,7+8\mathbb{Z}_2\] of index 2 between $(\mathbb{Z}_2^\times)^2 = 1 + 8\mathbb{Z}_2$ and $\mathbb{Z}_2^\times = 1 + 2\mathbb{Z}_2$. Therefore $\operatorname{nrd}(\mathcal{M}_2^\times)$ attaining a nonsquare invertible value modulo 8 is not enough to ensure that $\operatorname{nrd}(\mathcal{M}_2^\times) = \mathbb{Z}_2^\times$, and we must calculate all possible values of $\operatorname{nrd}$ modulo 8.

Let $\alpha = t + xi + yj + zij \in \mathcal{M} = \mathcal{M}_{(1)}$. It follows that \[\operatorname{nrd}(\alpha) = t^2 + 2x^2 + 7Dy^2 + 6Dz^2,\] and we examine the values modulo 8. In the case that $D \equiv 0 \mod 8$, this becomes \[ t^2 + 2x^2 \mod 8,\] and for this to be a unit in $\mathbb{Z}_2$ we must have $t$ odd. Hence in this case the group $\operatorname{nrd}(\mathcal{M}_2^\times) = 1,3 + 8\mathbb{Z}_2$ and is of index 2 in $\mathbb{Z}_2^\times$. When $D \not\equiv 0 \mod 8$, however, the reduced norm of $1+2j$ is \[ 1 + 4D \mod 8\] and notice that for odd $D$ this value is always 5 modulo 8. Also, the reduced norm of $1+ij$ is \[ 1+6D \mod 8,\] which for $D \equiv 2,6$ mod $8$ turns out to be $5$ mod $8$. For $D \equiv 4$ mod $8$ notice that the reduced norm of $1+ij$ is \[ 1+7D \equiv 5 \mod 8.\] In summary,\[[\mathbb{Z}_2^\times:\operatorname{nrd}(\mathcal{M}_2^\times)]=\begin{cases}
  2 & \text{when $D \equiv 0 \mod 8$}\\
  1 & \text{otherwise.}
\end{cases}\] For $\mathcal{M}_{(2)}, \mathcal{M}_{(5)}$ and $\mathcal{M}_{(6)}$, as the reduced norms of the elements \[1+i,\hspace{0.5em} 1+j,\hspace{0.5em} 1+2j \hspace{0.5em}\text{ and }\hspace{0.5em} 1+ij\] are \[3,\hspace{0.5em} 1+7D,\hspace{0.5em} 1+4D\hspace{0.5em} \text{and}\hspace{0.5em} 1+6D\hspace{0.5em} \text{modulo}\hspace{0.5em} 8,\] respectively, one can verify that in each order there are at least two elements whose reduced norms modulo 8 are distinct and are among $3, 5$ and $7$, regardless of the congruence class of $D$ modulo 8. Hence the index $[\mathbb{Z}_2^\times: \operatorname{nrd}(\mathcal{M}_2^\times)]$ is always 1 in these cases. These results combined with Section \ref{eichler} prove Theorem \ref{volumes}.

\begin{remark}
  The analysis of this section complements the proof of Lemma 3.7 of \cite{JUNG2019160}, where the values of $\operatorname{nrd}(\mathcal{M}_2^\times)$ were computed modulo only up to 4. It is likely that this also caused the error in Humbert's formula as mentioned in \cite{Maclachlan_Reid_1991}. Nonetheless, the lemma of \cite{JUNG2019160} stays valid which can be seen by further computation.
\end{remark}

\subsection{Counting Prime Geodesic Surfaces}

Using the volume formulae of Theorem \ref{volumes}, it is possible to prove a prime geodesic theorem for surfaces in the hyperbolic 3-fold $\Gamma\backslash\mathbb{H}^3$ using a simple analytic lemma given below:

\begin{lemma}[\cite{JUNG2019160}, 3.10]\label{analyticlemma}
  For any integer $r$, the function w.r.t. $X$ \[\# \{D \equiv r \mod 2^a \mid F(D) < X\}\] is asymptotic to $\frac{C}{2^a}X$ as $X \to \infty$, where the constant $C = \prod_{p}\left(1-\frac{1}{p}+\frac{1}{p+\left(\frac{-2}{p}\right)}\right)$, and $F(D)$ is as in Theorem \ref{volumes}.
\end{lemma}

Let $\Pi(x)$ denote the number of primitive immersed totally geodesic surfaces in the orbifold $\Gamma\backslash \mathbb{H}^3$ having area less than $x$. As such surfaces are in correspondence with $F$-subgroups of $\Gamma$, we may write the function $\Pi(x)$ as \[\Pi(x) = \sum_{k = 1}^6\# \{D \mid \operatorname{vol}(\mathcal{M}_{(k)}^1/\{\pm 1\}) < x\}\] where by the results of Theorem \ref{volumes}, we can write each summand in the form of \[\# \{D \mid \operatorname{vol}(\mathcal{M}_{(k)}^1/\{\pm 1\}) < x\} = \sum_i \#\{D \equiv r_i \mod 2^{a_i} \mid c_i \cdot \pi F(D) < x\}\] for corresponding values of $r_i, a_i$ and $c_i$. 
Applying Lemma \ref{analyticlemma}, we have \[\# \{D \equiv r_i \mod 2^{a_i} \mid c_i \cdot \pi F(D) < x\} \sim \frac{C}{\pi 2^{a_i}c_i}x\] as $x \to \infty$, and the sum of the corresponding asymptotic coefficients yields \[\Pi(x) \sim \frac{45C}{16\pi}x.\] This proves the identity \ref{asymp}.

\bibliographystyle{alpha}
\bibliography{references}

@article{e5932844-3dc8-37ed-b4a8-b4d26ff81211,
 ISSN = {00029947, 10886850},
 URL = {http://www.jstor.org/stable/2155089},
 author = {D. G. James and C. Maclachlan},
 journal = {Transactions of the American Mathematical Society},
 number = {5},
 pages = {1989--2002},
 publisher = {American Mathematical Society},
 title = {Fuchsian Subgroups of Bianchi Groups},
 urldate = {2025-11-22},
 volume = {348},
 year = {1996}
}

@article{Vulakh_1991,
title={Classification of Maximal Fuchsian Subsgroups of Some Bianchi Groups},
volume={34},
DOI={10.4153/CMB-1991-067-5},
number={3},
journal={Canadian Mathematical Bulletin},
author={Vulakh, L. Ya.},
year={1991},
pages={417–422}
}

@article{JUNG2019160,
title = {On the growth of the number of primitive totally geodesic surfaces in some hyperbolic 3-manifolds},
journal = {Journal of Number Theory},
volume = {202},
pages = {160-175},
year = {2019},
issn = {0022-314X},
doi = {https://doi.org/10.1016/j.jnt.2019.01.012},
url = {https://www.sciencedirect.com/science/article/pii/S0022314X19300514},
author = {Junehyuk Jung}
}

@article{Maclachlan_Reid_1991,
title={Parametrizing Fuchsian Subgroups of the Bianchi Groups},
volume={43},
DOI={10.4153/CJM-1991-009-1},
number={1},
journal={Canadian Journal of Mathematics},
author={Maclachlan, C. and Reid, A. W.},
year={1991},
pages={158–181}
}

@Inbook{Voight1,
author="Voight, John",
title="Volume formula",
bookTitle="Quaternion Algebras",
year="2021",
publisher="Springer International Publishing",
address="Cham",
pages="731--741",
isbn="978-3-030-56694-4",
doi="10.1007/978-3-030-56694-4_39",
url="https://doi.org/10.1007/978-3-030-56694-4_39"
}

@Inbook{Voight2,
author="Voight, John",
title="Quaternion orders: second meeting",
bookTitle="Quaternion Algebras",
year="2021",
publisher="Springer International Publishing",
address="Cham",
pages="393--411",
isbn="978-3-030-56694-4",
doi="10.1007/978-3-030-56694-4_24",
url="https://doi.org/10.1007/978-3-030-56694-4_24"
}

@inbook{inbook,
author = {Vulakh, L.},
year = {2017},
month = {10},
pages = {297-310},
title = {Maximal Fuchsian Subgroups of Extended Bianchi Groups},
isbn = {978-0-8247-8902-2},
doi = {10.4324/9780203747018-27}
}
\end{document}